\documentclass[11pt]{amsart}

\usepackage{amssymb}
\usepackage{hyperref}
\usepackage{multicol}
\usepackage{colordvi}
\usepackage{graphicx}
\usepackage{pgf,epsfig, epsf}
\numberwithin{equation}{section}
\theoremstyle{plain}
\newtheorem{Theorem}{Theorem}[section]
\newtheorem{corollary}[Theorem]{Corollary}
\newtheorem{proposition}[Theorem]{Proposition}

\theoremstyle{definition}
\newtheorem{remark}[Theorem]{Remark}
\newtheorem{example}[Theorem]{Example}
\newtheorem{definition}[Theorem]{Definition}

\copyrightinfo{}{}
\newcounter{FNC}[page]
\def\fauxfootnote#1{{\addtocounter{FNC}{2}\Magenta{$^\fnsymbol{FNC}$}%
     \let\thefootnote\relax\footnotetext{\Magenta{$^\fnsymbol{FNC}$#1}}}}

\newcommand{\conv}{\operatorname{conv}}

\newcommand{\w}{\operatorname{w}}
\newcommand{\nls}{\operatorname{nls_\Delta}}
\newcommand{\ls}{\operatorname{ls_\Delta}}
\newcommand{\lss}{\operatorname{ls_\square}}

\title{Bounds on Area Involving Lattice Size}

\author{Jenya Soprunova}
\address{Department of Mathematical Sciences\\
        Kent State University\\
        800 E. Summit st., Kent, OH 44242, USA}
\email{esopruno@kent.edu}
\urladdr{http://www.math.kent.edu/~soprunova/}


\begin{document}

\maketitle

\begin{abstract} The lattice size of a lattice polygon $P$ is a combinatorial invariant of $P$ that was recently   introduced  in relation to the problem of bounding the total degree and the bi-degree of the defining equation of an algebraic curve. In this paper, we establish sharp lower bounds on the area of plane convex bodies $P\subset\mathbb{R}^2$ that involve the lattice size of $P$. 
In particular, we improve bounds given by Arnold, and B\'ar\'any and Pach. We also provide a classification of minimal lattice polygons $P\subset\mathbb{R}^2$ of fixed lattice size $\lss(P)$.
\end{abstract}

\section{Introduction}
This paper is devoted to providing sharp lower bounds on the area of plane convex bodies $P$,  which involve the lattice size of $P$.
This invariant was formally introduced by Schicho, and Kastryck and Cools in~\cite{CasCools,Schicho}, although it had appeared implicitly earlier  in the work of Arnold~\cite{Arnold}, B\'ar\'any and Pach~\cite{BarPach}, Brown and Kaspzyck~\cite{BrownKasp}, and Lagarias and Ziegler~\cite{LagZieg}. The lattice size was further studied in~\cite{AlSopr, REU, HarSopr},  and~\cite{HarSoprTier}.

We next reproduce the definition of the lattice size from~\cite{CasCools} applying it now to a plane convex body $P$.  
\vspace{.2cm}
 \begin{definition}~\label{D:ls_intro} The {\it lattice size} $\operatorname{ls}_X(P)$ of  a convex body $P\subset\mathbb{R}^2$ with respect to a set $X\subset\mathbb{R}^2$  is the smallest  real non-negative $l$ such that 
 $\varphi(P)$ is contained in the $l$-dilate $lX$ of $X$ for some  transformation $\varphi$, which a combination of multiplication by a unimodular matrix and a translation by an integer vector. 
 \end{definition}

When $X=[0,1]\times\mathbb{R}$, the lattice size of $P$ with respect to $X$ coincides with the lattice width $\w(P)$ of $P$, an important invariant in convex geometry and its applications. Two other interesting 
invariants of $P$, denoted by $\ls(P)$ and $\lss(P)$, arise when $X$ is  the standard 2-simplex $\Delta=\conv\{(0,0), (1,0), (0,1)\}$ or the unit square $\square=[0,1]^2$.

It was shown in~\cite{HarSopr, HarSoprTier}  that in dimension 2 
a so-called reduced basis  computes both $\lss(P)$ and $\ls(P)$, and that in dimension 3 it computes $\lss(P)$, but not necessarily $\ls(P)$.
One can then use the generalized basis reduction algorithm, described and analyzed in~\cite{HarSopr,KaibSchnorr,LovScarf} 
to find the lattice size of a lattice polygon, which, as explained in~\cite{HarSopr, HarSoprTier}, outperforms  the ``onion skins'' algorithm of~\cite{CasCools, Schicho}. See Definition~\ref{D:reduced_basis} and Theorem~\ref{T:reduced_computes} for the definition of a reduced basis and for the precise formulation of the described results from~\cite{HarSopr} and~\cite{HarSoprTier}.

One of the questions that we address in this paper is the following: What is the smallest  possible nonzero area $A(P)$ of a lattice polygon $P$ of fixed lattice size $\lss(P)$ or $\ls(P)$? 
In Theorem~\ref{T:ls(P)2A} we prove a sharp bound $A(P)\geq \frac{1}{2}\ls(P)$   and describe the lattice polygons  on which this bound is attained. Since $\ls(P)\geq \lss(P)$ it follows that $A(P)\geq \frac{1}{2}\lss(P)$, and we show in Corollary~\ref{C:lss(P)2A} that this bound is sharp.

In the last section of the paper we  classify  inclusion-minimal lattice polygons $P$ with fixed lattice size  $\lss(P)$, see Theorem~\ref{T:minimal}. This classification provides an alternative proof of Corollary~\ref{C:lss(P)2A}.
Note that a classification of inclusion-minimal lattice polygons $P$ with fixed lattice width $\w(P)$ was provided in~\cite{CoolsLemmens}.

In both of the above bounds it is crucial that $P$ is a lattice polygon, since a plane convex body $P$ of fixed lattice size $\ls(P)$ or $\lss(P)$ may have an arbitrarily small area.
Hence, in the case of plane convex bodies,  it makes sense to look for lower bounds on the area that involve one of  the lattice sizes, $\ls(P)$ or $\lss(P)$, together with the lattice width $\w(P)$ of $P$.
In the case of the lattice size with respect to the unit square, such a bound was essentially proved by Fejes-T\'oth and Makai in~\cite{TothMakai}, where they showed  that for a plane convex body $P$ one has 
$A(P)\geq \frac{3}{8}{\w(P)}^2$, and that this bound is attained at $\conv\{(0,0), \left(w, \frac{w}{2}\right), \left(\frac{w}{2},w\right)\}$, where $w=\w(P)$. By a simple rescaling argument we establish in Theorem~\ref{T:whbound} a sharp bound $A(P)\geq\frac{3}{8}\w(P) \lss(P)$.

In our main result, Theorem~\ref{T:wlbound}, we establish a version of the Fejes-T\'oth--Makai result where we bound the area in terms of $\ls(P)$ and $\w(P)$. 
For  a plane convex body  $P$ we show that $A(P)\geq \frac{1}{4}\w(P)\ls(P)$ and describe convex bodies on which this bound is attained.  

The idea of inscribing a lattice polygon inside a small multiple of the unit square had appeared in~\cite{Arnold,BarPach, BrownKasp,LagZieg}, before the lattice size was introduced formally.
Both~\cite{Arnold} and~\cite{BarPach} are devoted to estimating the order of the number of lattice polygons of given area, up to the lattice equivalence, with~\cite{BarPach} improving the result of~\cite{Arnold}.  
It is shown in one of the steps of the argument in~\cite{Arnold} that for any lattice convex polygon $P\subset\mathbb{R}^2$ of nonzero area $A(P)$ there exists its lattice-equivalent copy inside a square of size $36A(P)$. In terms of the lattice size this means that $\lss(P)\leq 36A(P)$. Hence our result in Corollary~\ref{C:lss(P)2A} improves Arnold's bound from~\cite{Arnold} to a sharp one replacing a constant of 36 with 2.

A  similar result is proved  in \cite[Lemma 3]{BarPach}: For a convex lattice polygon $P$ with nonzero area $A(P)$  there exists a lattice-equivalent copy of $P$ inside a rectangle $[0,w]\times[0,h]$ with $wh<4 A(P)$.
In Theorem~\ref{T:whbound} we improve the bound of B\'ar\'any and Pach to a  sharp bound $A(P)\geq\frac{3}{8}\w(P) \lss(P)$. Note that our result from Theorem~\ref{T:ls(P)2A} can also be reformulated in  the spirit of~\cite{Arnold, BarPach}: For any lattice polygon $P$ of nonzero area $A(P)$ there exists a lattice-equivalent copy of $P$ contained in $2A(P)\Delta$.

%

\section{Definitions}
Recall that a {\it plane convex body} $P\subset\mathbb{R}^2$ is a  compact convex subset of $\mathbb{R}^2$ with non-empty interior. 
Given $(a,b)\in\mathbb{R}^2$, {\it the width of $P$ in the direction $(a,b)$} is 
$$\w_{(a,b)}(P)=\max\limits_{(x,y)\in P} (ax+by)-\min\limits_{(x,y)\in P} (ax+by).
$$
Consider the Minkowski sum of $P$ with $-P$, its reflection in the origin, and let $K:=(P+(-P))^{\ast}$ be the polar dual of the sum.
Then $K$ is origin-symmetric and convex and it defines a norm on $\mathbb{R}^2$ by
$$\Vert u\Vert _K=\inf\{\lambda>0\mid u/\lambda \in K\}.$$
For details see, for example, \cite{Barvinok}. We then have
$$\Vert u\Vert_K =\inf\{\lambda>0\mid u\cdot x\leq\lambda {\rm\ for\ all\ }  x\in K^\ast\}=\max\limits_{x\in K^{\ast}} u\cdot x =\frac{1}{2}\w_u(K^{\ast})=\w_u(P).
$$
This, in particular, implies that $u\mapsto \w_u(P)$ is a convex function on $\mathbb{R}^2$.

Recall that a vector $u=(a,b)\in\mathbb{Z}^2$ is called {\it primitive} if $\gcd(a,b)=1$. The {\it lattice width} of $P$, denoted by $\w(P)$, is the minimum of $\w_{u}(P)$ over all primitive directions $u$. 
 
A  {\it lattice polygon} is a convex polygon all of whose vertices have integer coordinates. 
  An integer square matrix $A$ is called {\it unimodular} if $\det A=\pm 1$. 
 Two convex bodies in $\mathbb{R}^2$ are called {\it lattice-equivalent} if one of them is the image of the other under a map which is a composition
 of multiplication by a unimodular matrix and a translation by an integer vector.

 Let $\Delta=\conv\{(0,0), (1,0), (0,1)\}\subset\mathbb{R}^2$ be the standard 2-simplex. 
 
 \begin{definition} The {\it lattice size} $\ls(P)$ of a convex body $P\subset\mathbb{R}^2$ with respect to the standard simplex  is the smallest $l\geq 0$ such that 
 the $l$-dilate $l\Delta$  contains a lattice-equivalent copy of $P$.
 \end{definition}
 
 Let $\square=[0,1]^2\subset\mathbb{R}^2$ be the unit square. 
 
 \begin{definition} The {\it lattice size} $\lss(P)$ of a convex body $P\subset\mathbb{R}^2$ with respect to the unit square   is the smallest $l\geq 0$ such that 
 the $l$-dilate $l\square$  contains a lattice-equivalent copy of~$P$.
 \end{definition}

 \begin{figure}[h]
\begin{center}
\includegraphics[scale=.7]{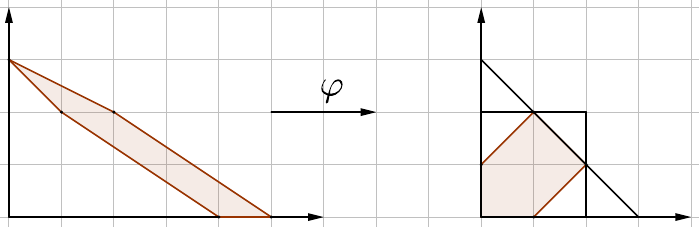} 
\caption{Example~\ref{E:lattice-size}}
\label{F:ls_example} 
\end{center}
\end{figure}

  \begin{example}\label{E:lattice-size} 
  Consider  $P=\conv\{(4,0), (5,0), (2,2), (0,3), (1,2)\}$, depicted in Figure~\ref{F:ls_example}. Let 
 $\varphi(x,y)=\begin{bmatrix}1&1\\-1&-2\end{bmatrix}\cdot \begin{bmatrix}x\\ y\end{bmatrix}+\begin{bmatrix}-3\\ 6\end{bmatrix}$. Then 
 $$\varphi(P)=\conv\{(1,2), (2,1), (1,0), (0,0), (0,1)\}.$$ Hence  $\varphi(P)\subset 2\square$ and
 we conclude that $\lss(P)=2$. Also,
 since $\varphi(P)\subset 3\Delta$ and $P$ has an interior lattice point while $2\Delta$ does not, it is impossible to unimodularly map $P$ inside $2\Delta$, and hence we have $\ls(P)=3$. 
  \end{example}

  \begin{definition}\label{D:reduced_basis} A basis $(u_1, u_2)$ of the integer lattice $\mathbb{Z}^2\subset\mathbb{R}^2$ is called {\it reduced} with respect to a convex body $P\subset\mathbb{R}^2$ if 
 $\w_ {u_1}(P)\leq \w_{u_2}(P)$ and  $\w_{u_1\pm u_2}(P)\geq \w_{u_2}(P)$.
 \end{definition}
 
 A fast algorithm for finding a reduced basis with respect to a convex body $P\subset\mathbb{R}^2$ was given in~\cite{KaibSchnorr}.
 It was shown in~\cite{HarSopr} and~\cite{HarSoprTier} that if the standard basis is reduced then one can easily find $\lss(P)$ and $\ls(P)$, as we summarize in Theorem~\ref{T:reduced_computes} below.
 
 \begin{definition} Let $P\subset\mathbb{R}^2$ be a plane convex body. We define $\nls(P)$ to be the smallest $l\geq 0$ such that $\varphi(P)\subset l\Delta$ where $\varphi$ is the composition of a multiplication by a matrix of the form 
 $\begin{bmatrix}\pm 1&0\\0&\pm 1\end{bmatrix}$ and a translation by an integer vector. Equivalently, if we let
 \begin{eqnarray*}
l_1(P)&:=&\max\limits_{(x,y)\in P}(x+y)-\min\limits_{(x,y)\in P}x-\min\limits_{(x,y)\in P}y,\\
l_2(P)&:=&\max\limits_{(x,y)\in P} x+\max\limits_{(x,y)\in P}y-\min\limits_{(x,y)\in P}(x+y),\\
l_3(P)&:=&\max\limits_{(x,y)\in P} y-\min\limits_{(x,y)\in P}x +\max\limits_{(x,y)\in P} (x-y),\\
l_4(P)&:=&\max\limits_{(x,y)\in P} x-\min\limits_{(x,y)\in P}y+\max\limits_{(x,y)\in P}(y-x),
\end{eqnarray*}
then $\nls(P)$ is the smallest of the four $l_i(P)$.
\end{definition}

\begin{example}\label{E:ex-l_i}
Let $P=\conv\{(0,0), (0,3), (2,2), (1,3)\}$, depicted in Figure~\ref{F:def-l_i}.  Then $l_1(P)=4$, $l_2(P)=5$, $l_3(P)=3$, and $l_4(P)=5$. Hence $\nls(P)=l_3(P)=3$.
 \begin{figure}[h]
\begin{center}
\includegraphics[scale=.7]{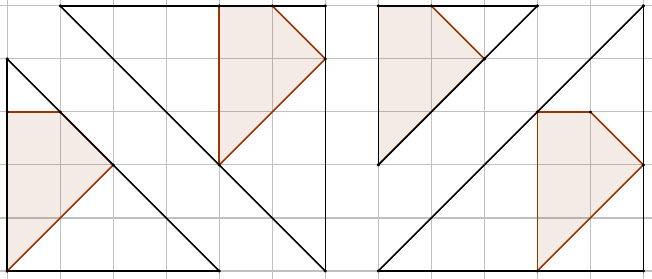} 
\caption{Example~\ref{E:ex-l_i}}
\label{F:def-l_i} 
\end{center}
\end{figure}
\end{example}

 \begin{Theorem}\label{T:reduced_computes}~\cite{HarSopr, HarSoprTier} Let $P\subset \mathbb{R}^2$ be a convex body. If the standard basis is reduced with respect to $P$, that is, $\w_{(1,0)}(P)\leq \w_{(0,1)}(P)$ and 
 $\w_{(1,\pm 1)}\geq \w_{(0,1)}(P)$, then $\ls(P)=\nls(P)$, $\lss(P)=\w_{(0,1)}(P)$, and $\w(P)=\w_{(1,0)}(P)$
  \end{Theorem}
  
For polygon $P$ in Example~\ref{E:ex-l_i} we have $\w_{(1,0)}(P)=2$, $\w_{(0,1)}(P)=3$, $\w_{(1,1)}(P)=4$, and $\w_{(1,-1)}(P)=3$. Hence the standard basis is reduced with respect to $P$, and  we conclude that  $\w(P)=2$, 
$\lss(P)=3$, and $\ls(P)=\nls(P)=3$.

\begin{example}\label{E:ex-I}
Let $I=[(0,0), (0,l)]$. Then $\w_{(1,0)}(I)=0$,  $\w_{(0,1)}(I)=l$, $\w_{(1,1)}(I)=l$,  and $\w_{(1,-1)}(I)=l$, so the standard basis is reduced. Hence by Theorem~\ref{T:reduced_computes} we have
$\ls(I)=\nls(I)=l$ and $\lss(I)=\w_{(0,1)}=l$. It follows that for any lattice segment $I$ containing $l+1$ lattice points we have $\lss(I)=\ls(I)=l.$ 
 \end{example} 
  
Note that for   $u\in\mathbb{R}^2$ and a unimodular matrix $A$ of size 2 we have  
$$\w_{u}(AP)=\max\limits_{x\in P} u\cdot (Ax)- \min_{x\in P} u\cdot (Ax)=\max\limits_{x\in P} (A^{T}u)\cdot x- \min_{x\in P} (A^{T}u)\cdot x=\w_{A^{T}u}(P).
$$
Therefore, if the rows of  $A$ are $u_1$ and $u_2$, then $\w_{(1,0)}(AP)=\w_{u_1}(P)$ and $\w_{(0,1)}(AP)=\w_{u_2}(P)$.
In the  next observation we show that if a plane convex body $P$ is inscribed in  $[0,h]^2$ and touches all of its sides, then  $\lss(P)=h$.

 \begin{proposition}\label{P:touch} Suppose that $\w_{(1,0)}(P)=\w_{(0,1)}(P)=h$. Then $\lss(P)=h$ and $$\w(P)=\min\{h,\w_{(1,1)}(P), \w_{(1,-1)}(P)\}.$$
 \end{proposition}
\begin{proof} Let $(a,b)\in\mathbb{Z}^2$ be a primitive direction. Without loss of generality we can assume that $|a|\geq |b|$. We then have
$$|a|h=\w_{(a,0)}(P)\leq \w_{(a,b)}(P)+\w_{(0,b)}(P)=\w_{(a,b)}(P)+|b|h,
$$
and so we get $\w_{(a,b)}(P)\geq (|a|-|b|)h$. Hence for $|a|>|b|$ we conclude that $\w_{(a,b)}(P)\geq h$. Hence directions $(\pm 1,\pm 1)$ are the only primitive directions with respect to which the width of $P$
could be less than $h$. Since no two out of these four directions can be used as rows to form a unimodular matrix, the conclusion follows.
\end{proof}

 \begin{definition}\label{D:drop} Let $P\subset\mathbb{R}^2$ be a lattice polygon and let $p\in P$ be one of its vertices. We say that a lattice polygon $Q$ is obtained from $P$ by $\it dropping$ $p$ if $Q$ is the convex hull of all the lattice points of $P$, except $p$.
 \end{definition}

\section{Lower bounds on the area of a plane convex body in terms of its width and lattice size.}

Consider a convex body $P\subset\mathbb{R}^2$. Let $w=\w(P)$ be its width and $A(P)$ be its area. It was shown in~\cite{TothMakai} that  $A(P)\geq \frac{3}{8}w^2$ and that this bound is attained only on the convex bodies that are lattice-equivalent to  $\conv\{(0,0), \left(w, \frac{w}{2}\right), \left(\frac{w}{2},w\right)\}$. We formulate a result which is a straight-forward corollary of this bound.

\begin{Theorem}\label{T:whbound}
Let $P\subset\mathbb{R}^2$ be a convex body with $h=\lss(P)$ and  $w=\w(P)$. Then for the area $A(P)$ of $P$ we have $A(P)\geq \frac{3}{8}wh$. This bound is sharp and attained only on the convex bodies $P$ that are lattice-equivalent to $\conv\{(0,0), \left(w, \frac{w}{2}\right), \left(\frac{w}{2},w\right)\}$.
\end{Theorem}

\begin{remark} It was shown in Lemma~3 of~\cite{BarPach} that for any lattice convex polygon $P\subset\mathbb{R}^2$ there exist numbers $w,h\geq 0$ with $wh\leq 4 A(P)$, such that  $[0,w]\times [0,h]$ 
contains a lattice-equivalent copy of $P$. Note that Theorem~\ref{T:whbound} strengthens this result replacing the constant of 4 with $\frac{8}{3}$, and also extends the result to plane convex bodies.
\end{remark}

\begin{proof}[Proof of Theorem~\ref{T:whbound}]
Let the standard basis be reduced with respect to $P$. Then by Theorem~\ref{T:reduced_computes} $\w_{(1,0)}(P)=w$ and $\w_{(0,1)}(P)=h$. Let $P'$ be the image of $P$ under the map $(x,y)\mapsto (x,\frac{w}{h}y)$.  
We next  check that the standard basis is also reduced with respect to $P'$. Denote by $(x_1,y_1)$ and $(x_2,y_2)$ points in $P$ that, correspondingly, minimize and maximize $x+y$ over $P$. Then  we have 
$\w_{(1,1)}(P)=x_2+y_2-(x_1+y_1)\geq h$. Since $y_2-y_1\leq h$, this implies that $x_2-x_1\geq 0$. Hence
$$\w_{(1,1)}(P')\geq (x_2-x_1)+\frac{w}{h}(y_2-y_1)\geq \frac{w}{h}(x_2-x_1)+\frac{w}{h}(y_2-y_1)\geq w.
$$  
Similarly, first reflecting $P$ in the line $y=w/2$, we conclude that $\w_{(1,-1)}(P')\geq w$, so we have checked that the standard basis remains reduced as we pass from $P$ to $P'$. 
We conclude that  $\w(P')=w$, and hence, as shown in~\cite{TothMakai}, we have  $A(P')\geq \frac{3}{8}w^2$ and this bound is attained at $P'$ which are lattice-equivalent to $\conv\{(0,0), \left(w, \frac{w}{2}\right), \left(\frac{w}{2},w\right)\}$.
This implies that $A(P)=\frac{h}{w}A(P')\geq \frac{3}{8}wh$ and this  bound can be attained only if  $P$ is lattice-equivalent to $\conv\{(0,0), \left(w, \frac{h}{2}\right), \left(\frac{w}{2},h\right)\}$. Since 
the standard basis is reduced with respect to $P$, if $w\geq h/2$, we have $\w_{(1,-1)}(P)=\frac{w+h}{2}\geq h$, which implies $w=h$. If $w<h/2$ then $\w_{(1,-1)}=h-w/2\geq h$, which implies $w=0$. We conclude that the bound is attained exactly at $P$ which are lattice-equivalent to $\conv\left\{(0,0), \left(w, \frac{w}{2}\right), \left(\frac{w}{2},w\right)\right\}$.
\end{proof}
 
 Let $P$ be a convex body with $l=\ls(P)$,  $h=\lss(P)$, and $w=\w(P)$. We can assume that $P\subset [0,h]^2$ and hence $P\subset[0,h]^2\subset 2h\Delta$, so we conclude that $\ls(P)\leq 2\lss(P)$.
 This implies that  $A(P)\geq \frac{3}{8}wh\geq \frac{3}{16}wl$. We now  improve this bound to a sharp bound $A(P)\geq \frac{wl}{4}$.

\begin{Theorem}\label{T:wlbound}
Let $P\subset\mathbb{R}^2$ be a convex body with $l=\ls(P)$ and $w=\w(P)$. Then for the area $A(P)$ of $P$ we have $A(P)\geq \frac{wl}{4}$. This bound is attained only at $P$ which are lattice-equivalent to 
$\conv\left\{(0,0), \left(w,\frac{w}{2}\right),\left(\frac{w}{2}, w\right)\right\}$.
\end{Theorem}

\begin{proof}
Let the standard basis be reduced with respect to $P$. Then by Theorem~\ref{T:reduced_computes} we have  $w=\w_{(1,0)}(P)$ and $h:=\lss(P)=\w_{(0,1)}(P)$, so we can assume $P\subset \Pi:= [0,w]\times[0,h]$. We can also assume  that $P\subset l\Delta$, where $l=\ls(P)$ and $P$ touches all three sides of $l\Delta$. 

Suppose first that  $h=w$ and hence $P\subset\Pi=[0,w]^2$. Pick points $p_1, p_2, p_3$, and $p_4$ in $P$, one on each side of $\Pi$, and points $q_1, q_2, q_3$, and $q_4$ that maximize and minimize over $P$ the linear functions  $x+y$ and $x-y$, as depicted in the first diagram of Figure~\ref{F:8to4points}. Note that some of these eight points may coincide. Let $Q=\conv\{p_1, p_2, p_3, p_4, q_1, q_2, q_3, q_4\}$. 
\begin{figure}
\begin{center}
\includegraphics[scale=.65]{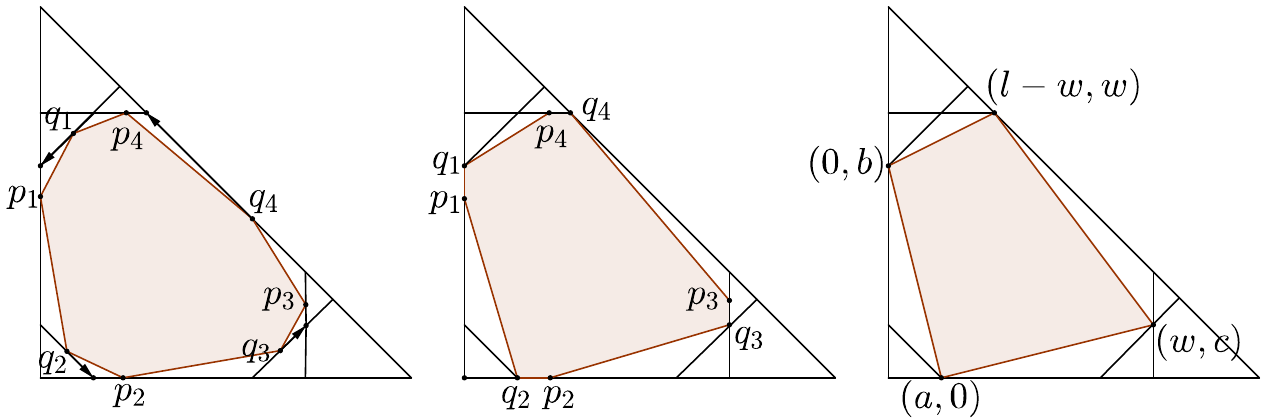} 
\caption{Case 0 reduction}
\label{F:8to4points} 
\end{center}
\end{figure}

If we move $q_4$ within $\Pi$ along the support line $x+y=l$ , the area of the triangle with the vertices $p_3, q_4, p_4$ will be the smallest
when $q_4$ is on $y=w$ or $x=w$, depending on the slope of the line connecting $p_3$ and $p_4$. Therefore, we can move $q_4$ to the boundary of $\Pi$ not increasing the area, and preserving the width and both lattice sizes.
Similarly, we move $q_1$, $q_2$, and $q_3$ along the corresponding support lines to the boundary of $\Pi$. If we end up with the case when there is one of the  $q_i$ on each side of $\Pi$, as in the second diagram of Figure~\ref{F:8to4points}, we pass to $R=\conv\{q_1,q_2, q_3, q_4\}$, depicted in the third diagram. 

Note that passing from $P$ to $Q$ to $R$ we did not increase the area and did not change  the minima and the maxima in the directions $(1,0), (0,1), (1,\pm 1)$.  Hence the standard basis remains reduced and there is no change in $l_1, l_2,l_3$ and $l_4$. Hence $R$ has the same $l,w$, and $h$ as $P$, and $\ls(R)=l_1(R)=l$.

 Let $q_1=(0,b)$, $q_2=(a,0)$, $q_3=(w,c)$, and $q_4=(l-w,w)$. Note that since the standard basis is reduced we have $\w_{(1,1)}(R)=l-a\geq w$, so $l-w-a\geq 0$. Also, $\w_{(1,-1)}(R)=w-c+b\geq w$, so $b\geq c$. We get
\begin{align*}
2A(P)\geq 2A(R)&=2w^2-ab-(w-b)(l-w)-(w-a)c-(2w-l)(w-c)\\
&=w^2+(b-c)(l-w-a)\geq w^2\geq\frac{wl}{2},
\end{align*}
where we used $2w\geq l$ which holds true since $(w,w)$ is on $x+y=l$ or outside of $l\Delta$. In order for the inequality to be attained we would need $w=l/2$ and also $b=c$ or $l=w+a$.
Let $w=l/2$ and  $b=c$. Then we  have $l_4(R)=w+b\geq l$  and $c+w\leq l$ since $(c,w)\in l\Delta$. Hence $b=c=l/2$, but this contradicts $l_3(R)=2w-c\geq l$.
If $w=l/2$ and $l=w+a$, then $a=l/2$ and this contradicts $l_4(R)=2w-a\geq l$. We conclude that in this case the inequality is strict.

Each of the $q_i$'s may slide along the support lines in one of the two directions. Due to the symmetry in the line $x=y$, we can assume that $q_4$ slides toward $y=w$, as  in Figure~\ref{F:8to4points}. Since for each of $q_1, q_2, q_3$ we have have two choices, there are eight cases total, one of which we just covered and will refer to  as Case 0.
The remaining seven cases are depicted in  Figure~\ref{F:SevenCases}, where we also introduce the notation for the coordinates for some of the $q_i$'s. In each of these cases, we move the $q_i$'s to the boundary of $\Pi$ not increasing the area, after which we drop (see Definition~\ref{D:drop}) the $p_i$'s on the sides  of $\Pi$ where we now have a $q_i$. We next cover each of these seven cases.   
\begin{figure}
\begin{center}
\includegraphics[scale=.85]{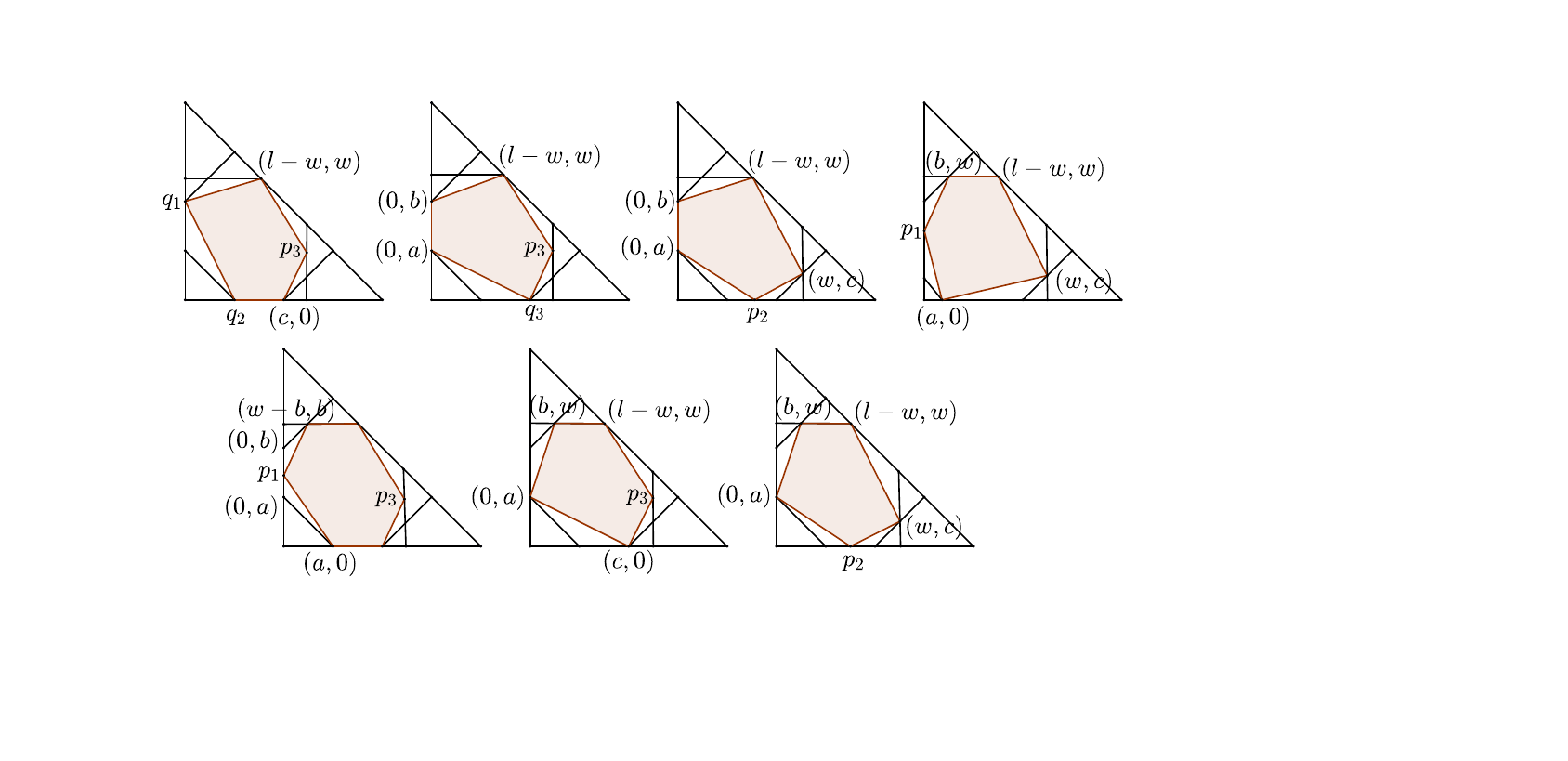} 
\caption{Seven cases}
\label{F:SevenCases} 
\end{center}
\end{figure}

In Case 1, let $q_3=(c,0)$ and $q_4=(l-w,w)$, as depicted in the first diagram in Figure~\ref{F:SevenCases}. Then  we have $l_3(R)=w+c\geq l$, and hence the  slope of the line connecting $(l-w,w)$ to $(c,0)$ is negative. This implies that we can slide $p_3$ to $(w,w-c)$ and then, unless $c=w$,  drop $(c,0)$, which reduces this case to Case 0.
 
In Case 2, $l_2(R)=2w-a\geq l$ implies that  the slope of the line connecting $(0,a)$ to $(l-w,w)$ is at least 1, so we can slide $(0,b)$ to $(w-b,w)$ not increasing the area.
This reduces Case 2 to Case 6.

\begin{figure}[h]
\begin{center}
\includegraphics[scale=.85]{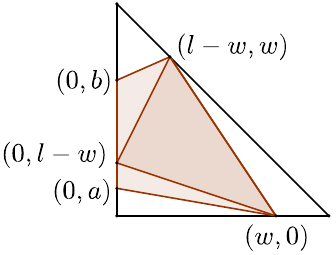} 
\caption{Case 3}
\label{F:RST} 
\end{center}
\end{figure}

In Case 3,  if $a\leq c$, we can slide $p_2$  to $(a,0)$ and,  unless $a=0$, drop $(0,a)$ reducing to Case 0.
If $a>c$, we move $p_2$ all the way to $(w,0)$ and, unless $c=0$, drop $(w,c)$. The obtained quadrilateral $S$ is depicted in Figure~\ref{F:RST}. We have $A(S)\leq A(R)$, $\w_{(1,-1)}(S)\geq \w_{(1,-1)}(R)$, and
$l_3(S)\geq l_3(R)$, while all other parameters remain the same. Hence the standard basis remains reduced and $l_1(S)=l_1(R)=\ls(S)$.

We have
\begin{align*}
2A(S)=2w^2-aw-(w-b)(l-w)-w(2w-l)=w(w-a)+b(l-w).
\end{align*}
 Since $l_2(S)=2w-a\geq l$ we get $2A(S)\geq w(l-w)+b(l-w)=(w+b)(l-w)$. Also, $l_4(S)=w+b\geq l$ and we can conclude that $2A(S)> \frac{wl}{2}$, provided that $l> \frac{3}{2}w$. 
 
 Suppose next that $l\leq\frac{3}{2}w$.
We have $l_4(S)=w+b\geq l$ and $\w_{(1,1)}(P)=l-a\geq w$ and hence $a\leq l-w\leq b$. We conclude that  $T=\conv\{(0,l-w), (w,0), (l-w,w)\}$ is contained in $S$ and we get 

 \begin{align*}
2A(P)\geq 2A(T)&=\left|\begin{matrix}w&-(l-w)\\l-w&2w-l\end{matrix}\right|=w(2w-l)+(l-w)^2\\
&=w^2-(l-w)(2w-l)\geq  w^2-\frac{w}{2}(2w-l)=\frac{wl}{2},
\end{align*}
where we used $l_3(S)=2w\geq l$, so $2w-l\geq 0$.
The inequality turns into equality if and only if $l=\frac{3}{2}w$ and $P$ is lattice-equivalent to $T$, that is, to $\conv\left\{\left(w,0\right), \left(\frac{w}{2}, w\right), \left(0,\frac{w}{2}\right)\right\},$ which is lattice-equivalent to
$\conv\{(0,0), \left(\frac{w}{2}, w\right), \left(w,\frac{w}{2}\right)\}$ via the map $(x,y)\mapsto(w-x,y)$.
 
In Case 4, if $b\leq a$  we move  $p_1$ to $(0,w-b)$ and  drop $(b,w)$ to reduce to Case 0.
\begin{figure}[h]
\begin{center}
\includegraphics[scale=.85]{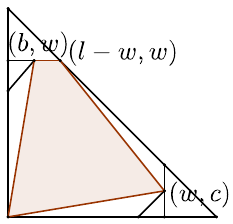} 
\caption{Case 4}
\label{F:Case4} 
\end{center}
\end{figure}
If $b>a$ we move $p_1$ all the way to $(0,0)$ and, unless $a=0$, drop $(a,0)$. The obtained quadrilateral $S$ is depicted in Figure~\ref{F:Case4}. If  $3w<2l$ then  
 \begin{align*}
2A(S)&=2w^2-bw-wc-(w-c)(2w-l)=w(w-b)+(w-c)(l-w)\\
&\geq w(l-w)+(w-c)(l-w)=(l-w)(2w-c)> \frac{wl}{2},
\end{align*} 
where we used $l_3(S)=2w-c\geq l$ together with $l-w>w/2$.

If $3w\geq 2l$, move $(w,c)$ to $(w,l-w)$, non-increasing the area, and let $T=\conv\{(0,0), (w,l-w), (l-w,w)\}.$ Then
$$2A(P)\geq 2A(S)\geq  2A(T)=2wl-l^2=l(2w-l)\geq \frac{wl}{2}.
$$
The inequality turns into equality if and only if $3w=2l$ and $P$ is lattice-equivalent to $T$, that is,  to $\conv\{(0,0), (\frac{w}{2},w), (w,\frac{w}{2})\}$.

In Case 5, if $w-b\leq a$, we move $p_1$ to $(0,b)$ and remove $(w-b,w)$, reducing this case to Case 1. If $w-b> a$, we move $p_1$  to $(0,a)$  and remove $(a,0)$, reducing this case to Case 6.
 
In Case 6, we have $l_3(R)=w+c\geq l$, and hence the area does not increase as we move $p_3$ to $(w,0)$ and then drop $(c,0)$, unless $c=w$, to get $S$. Then
$$2A(S)=2w^2-aw-(w-a)b-(2w-l)w=(w-a)(w-b)+w(l-w).
$$
Since $l_4(S)=2w-b\geq l$ we get $w-b\geq l-w$ and hence
$$2A(S)\geq (w-a)(l-w)+w(l-w)=(2w-a)(l-w).
$$
We have $l_2(S)=2w-a\geq l$ and hence, if we additionally assume that $l>\frac{3}{2}w$, we get $2A(R)> (2w-a)(l-w)> \frac{wl}{2}$.

Assume next that $l\leq \frac{3}{2}w$. We have $l_3(S)=w+c\geq l$ and $\w_{(1,1)}(S)=l-a\geq w$.
Hence, dropping $(b,w)$, unless $b=l-w$, and then moving  $(0,a)$ up to $(0,l-w)$  
we obtain triangle $T=\conv\{(0,l-w), (w,0), (l-w,w)\}$, considered in Case 3.

In Case 7, if $c\leq a$, we can move $p_2$ to $(w-c,0)$  and we are in Case 6. If  $c>a$ we move $p_2$ to $(a,0)$ and end up in Case 4, which completes the argument.

It remains to explain how the general case reduces to the case when $h=w$.
Let $P'$ be the image of $P$ under the map $(x,y)\mapsto (x,\frac{w}{h}y)$.
Then, as we have shown in the proof of Theorem~\ref{T:whbound},  the standard basis is also reduced with respect to $P'$. Hence $\w(P')=\lss(P')=w$ and 
$\ls(P')$ is the smallest of $l_1(P'), l_2(P'), l_3(P')$, and $l_4(P')$. Pick  $(x_1,y_1)\in P$ that satisfies $x_1+y_1=l$.
Then $(x_1,\frac{w}{h}y_1)\in P'$ and we have
$$l_1(P')\geq x_1+\frac{w}{h}y_1\geq\frac{w}{h}(x_1+y_1)=\frac{w}{h}l_1(P)=\frac{w}{h}l.
$$
Similarly, for $i=2,3,4$ we get $l_i(P')\geq \frac{w}{h}l_i(P)\geq \frac{w}{h} l_1(P)=\frac{w}{h}l$ and hence we can conclude that $\ls(P')\geq \frac{w}{h}l$. 

From the above argument we know that $A(P')\geq\frac{\ls(P')\w(P')}{4}$. Together with $\ls(P')\geq \frac{w}{h}l$ and $A(P')=\frac{w}{h}A(P)$ this implies $A(P)\geq \frac{wl}{4}$. 
The inequality for $P'$ turns into equality if and only if $P'=\conv\left\{(0,0), \left(w,\frac{w}{2}\right),\left(\frac{w}{2}, w\right)\right\}$. For such $P'$ we have
$$P=\conv\left\{(0,0), \left(w,\frac{h}{2}\right),\left(\frac{w}{2}, h\right)\right\}.$$ 
We get  $A(P)=\frac{3}{8}wh$,  $\ls(P)=l_1(P)=\frac{w}{2}+h$ and $\w(P)=w$. It follows that $A(P)\geq \frac{\w(P)\ls(P)}{4}$ turns into equality only if $\frac{3}{8}wh=\frac{w}{4}\left(\frac{w}{2}+h\right)$, which is equivalent to $h=w$.

\end{proof}

\section{Proving $A(P)\geq \frac{1}{2}\ls(P)$  for lattice polygons $P$.}

If $P\subset\mathbb{R}^2$ is a lattice polygon with nonzero area then its width is at least 1 and Theorem~\ref{T:wlbound} implies that $A(P)\geq \frac{1}{4}\ls(P)$. In this section we will improve this bound to a sharp bound $A(P)\geq \frac{1}{2}\ls(P)$. We will also observe that this implies that $A(P)\geq \frac{1}{2}\lss(P)$, which is again a sharp bound.

\begin{Theorem}\label{T:ls(P)2A} Let $P\subset\mathbb{R}^2$ be a convex lattice polygon of nonzero area $A(P)$ and lattice size $l=\ls(P)$. Then $A(P)\geq \frac{1}{2}\ls(P)$. This bound is sharp and is attained  exactly at lattice polygons $P$ that are lattice-equivalent to one of the following:
\begin{itemize}
\item[(a)] $\conv\{(0,0), (l,0), (0,1)\}$ for $l\geq 1$;
\item[(b)] $[0,1]^2$ for $l=2$;
\item[(c)] $T_0=\conv\{(0,0), (1,2), (2,1)\}$ for $l=3$.
\end{itemize}
\end{Theorem}
\begin{remark} It follows from this theorem that for any lattice polygon $P$ of nonzero area $A(P)$ there exists a lattice-equivalent copy of $P$ contained in $2A(P)\Delta$.
\end{remark}

\begin{proof}[Proof of Theorem~\ref{T:ls(P)2A}]
Let $l=l_1(P)=\ls(P)$ so that $P\subset l\Delta$.  For $l=1$ and 2 the conclusion is clear, so we will assume that $l\geq 3$. 
We first  consider the case where  $P$ contains one of the vertices of $l\Delta$. 
Note that  $\varphi: \begin{bmatrix} x\\y\end{bmatrix}\to \begin{bmatrix} -1&-1\\1&0\end{bmatrix}\begin{bmatrix} x\\y\end{bmatrix}+
\begin{bmatrix} l\\0\end{bmatrix}$ maps  $l\Delta$ to itself rotating its vertices in the counterclockwise direction.
Using map $\varphi$ together with the reflection in the line $y=x$ we can assume that  $P$ contains the origin and point $(c,l-c)$ with $l/2\leq c\leq l$.

Denote $I=[(0,0), (c,l-c)]$. Let $m=\max_{(x,y)\in P} x$ be attained at $p\in P$. Since 
$$l_4(P)=m+\max_{(x,y)\in P}(y-x)\geq l,$$ 
we have $\max_{(x,y)\in P}(y-x)\geq l-m$. Let this maximum be attained at $q\in P$. 

\begin{figure}[h]
\begin{center}
\includegraphics[scale=.65]{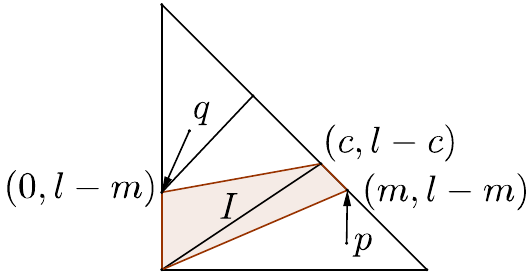} 
\caption{$P$ contains the origin}
\label{F:ls(P)2A}
\end{center}
\end{figure}

Let $Q=\conv\{(0,0), (c,l-c), p, q\}$.
Since the slope of $I$ is at most 1 we can move $p$ and $q$ to $(m,l-m)$ and $(0,l-m)$ correspondingly not increasing the area of $Q$, as illustrated in  Figure~\ref{F:ls(P)2A}. Therefore the area of $Q$
is bounded below by the area of  
$$R=\conv\{(0,0), (m,l-m), (0,l-m)\}.$$
For $m\leq l-2$ we  have $2A(P)\geq 2A(R)=m(l-m)\geq 2m\geq l$ since $m\geq c\geq l/2$.
This inequality turns into equality if and only if $m=c=2$ and $l=4$. Then $R=\{(0,0), (2,2), (0,2)\}$, $p=(2,2)$, and the only options for $q$ are $(0,2)$ and $(1,3)$. In both of these cases we get $\ls(R)<4$, 
and hence the inequality is not attained under our current assumptions.

If $m=l$, since $P$ has nonzero area, it contains at least one lattice point outside $[(0,0), (l,0)]$. Hence $A(P)\geq l/2$, with equality if and only if $P$ is lattice-equivalent to $\conv\{(0,0), (l,0), (0,1)\}$.

Let $m=l-1$ and assume that $(l-1,0)\in P$. If $P$ contains a point with the $y$-coordinate equal to at least 2, then $A(P)\geq l-1> l/2$  and hence we can assume that $P\subset [0,l-1]\times[0,1]$.  
The lattice size of any lattice polygon properly contained in  $[0,l-1]\times[0,1]$ is  less than $l$, and we conclude that  $P=[0,l-1]\times[0,1]$, so $A(P)=l-1> l/2$.

If $(l-1,0)\not\in P$ then $P$ contains $(l-1,1)$ and we have $\max_{(x,y)\in P}(x-y)=l-2$. Hence from $l_3(P)\geq l$ we get $\max_{(x,y)\in P} y\geq 2$
and $A(P)$ is larger than or equal to the area of $Q=\conv\{(0,0), (l-1,1), (l-2,2)\}$, whose area is $l/2$. Also, $A(P)>\frac{l}{2}$ unless $P=Q$.
Also, $l_4(Q)=l-1$ unless $l\leq 3$. Hence we get a strict bound $l(P)>l/2$ for $l>3$. For $l=3$ the bound is attained at $\{(0,0), (2,1), (1,2)\}$.

Now we can assume that $P$ contains a lattice triangle $T$ with exactly one vertex on each side of $l\Delta$, as depicted in Figure~\ref{F:triangle}.
Let the vertices of $T$ be $(a,0)$, $(0,b)$, and $(c,l-c)$, with $a,b,c\in[1,l-1]$. As above, using map $\varphi$ together with the reflection in the line $y=x$,  we can assume that $a$ is the smallest
of the six numbers $a,l-a,b,l-b,c,l-c$. This, in particular, implies that $a\leq b, a\leq c$, and $a\leq l/2$.
\begin{figure}[h]
\begin{center}
\includegraphics[scale=.65]{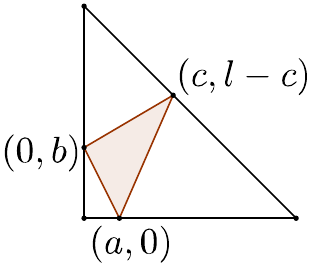} 
\caption{$P$ contains lattice triangle $T$}
\label{F:triangle}
\end{center}
\end{figure}
For $a\geq 2$ and $b\leq l/2$ we have
$$2A(T)=\det\begin{bmatrix} a&-b\\c&l-b-c\end{bmatrix}=al-ab-ac+bc=a(l-b)+c(b-a)\geq 2(l-b)\geq l.
$$
Note that we get equality  if and only if $a=b=2$, $l=4$, and $c=1,2$ or 3, but then $l_2(T)<l$. This means that $T\subsetneq P$ and hence $2A(P)>l$. 
If $a\geq 2$ and $c\leq l/2$ then $2A(T)=a(l-c)+b(c-a)\geq l$ with equality if and only if $a=c=2$, $l=4$, and $b=1,2,3$ in each of which case we have
$l_2(T)<l$, so we get the same conlclusion.

Suppose next that $b>l/2$, $c>l/2$, while $a\geq 2$. If we also have $c\geq a+2$ or $b\geq a+2$ then
$$2A(T)=b(c-a)+a(l-c)\geq 2b> l {\rm\ \  or\ \ } 2A(T)=c(b-a)+a(l-b)\geq 2c>l.
$$
If we have both $l/2<b\leq a+1$ and $l/2<c\leq a+1$ then since $a\leq l/2$ we have $b=c=a+1$. We get
$$2A(T)=l+(a-1)(l-1-a)\geq l.
$$
This inequality is in fact strict since $a=l-1$ would imply $b=c=l$ and since we assumed that $a\geq 2$. 

\begin{figure}[h]
\begin{center}
\includegraphics[scale=.32]{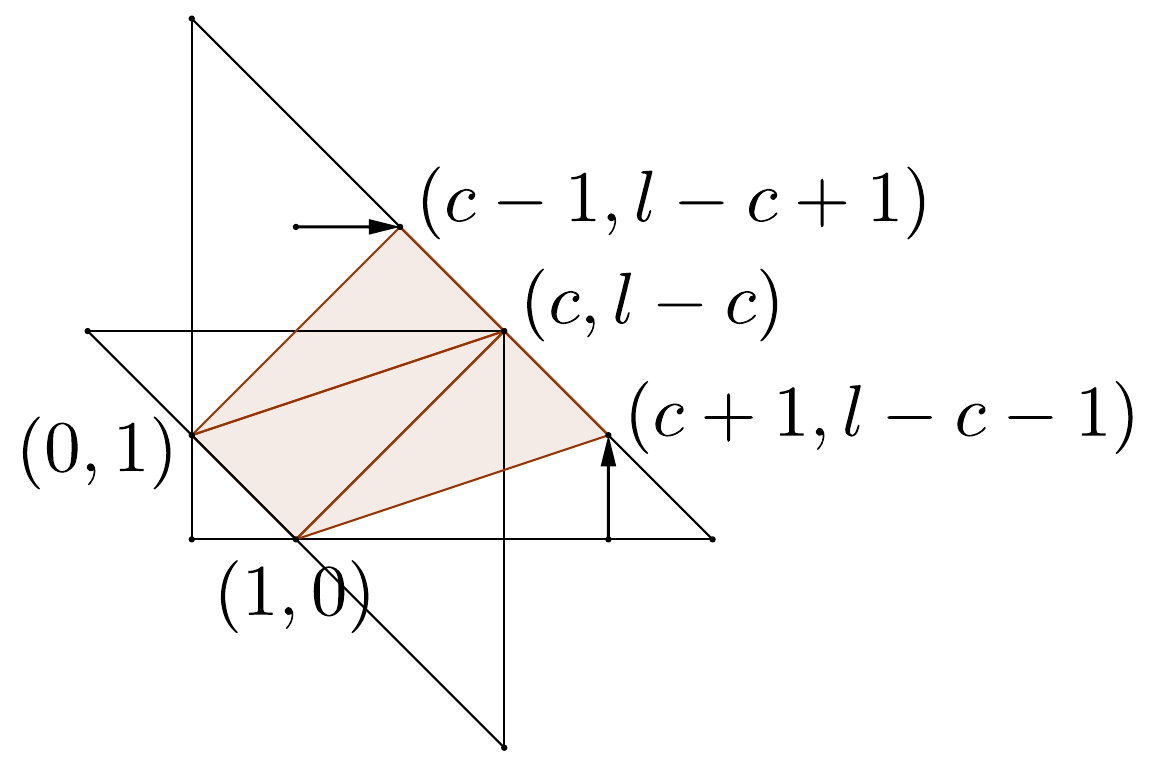} 
\caption{The case $a=b=1$}
\label{F:3triangles}
\end{center}
\end{figure}

Finally, if $a=1$ then $2A(T)=bc+l-b-c$, and we have  $2A(T)> l$ unless $bc\leq b+c$, that is, $b=c=2$ or one of $b,c$ is equal to 1. 
Let $b=c=2$ and assume that $l=3$. Then we see that the bound is attained  at a  triangle lattice-equivalent to $T_0$.
If $l\geq 4$ we have $2A(T)=l$ and $l_2(T)=l-1$, which means that $P$ contains $T$ properly and $2A(P)>l$.

If  $b=1$ then $2A(T)=l-1$ and  $l_2(T)=l-1$, which implies $T\subsetneq P$. Even if $(0,0)\in P$ we would still have $l_4(P)<l$ and hence $P$ has to contain
a lattice point  $(x,y)\in l\Delta$ such that $x\geq c+1$ or $y\geq l-c+1$. In the first of these two cases the area of $\conv\{T\cup(x,y)\}$ would be minimal if $(x,y)=(c+1,l-c-1)$, and 
in the second the area would be minimal if $(x,y)=(c-1,l-c+1)$.  See Figure~\ref{F:3triangles} for an illustration.
In both cases we get $2A(P)\geq 2(l-1)>l$.

If $c=1$ we have $l_2(T)=l-1$ and hence $P$ contains a point with the $x$-coordinate equal to at least 2, which implies $2A(P)\geq 2(l-1)>l$.
\end{proof}

\begin{corollary}\label{C:lss(P)2A}
Let $P\subset\mathbb{R}^2$ be a convex lattice polygon with nonzero area $A(P)$ and $\lss(P)=h$. Then $A(P)\geq \frac{1}{2}\lss(P)$ 
and this inequality turns into equality if and only if $P$ is lattice-equivalent to $\conv\{(0,0), (h,0), (0,1)\}.$ 
\end{corollary}
\begin{remark} It was shown in~\cite{Arnold} that any convex lattice  polygon $P\subset\mathbb{R}^2$ has a lattice-equivalent copy inside a square of size $36 A(P)$.
The corollary strengthens this result replacing the constant of 36 with 2.
\end{remark}
\begin{proof}[Proof of Corollary~\ref{C:lss(P)2A}]
If $\ls(P)=l$ then $P\subset l\Delta\subset [0,l]^2$ and we have $\lss(P)\leq l=\ls(P)$. Hence the result of Theorem~\ref{T:ls(P)2A} implies that $A(P)\geq  \frac{1}{2}\ls(P)\geq \frac{1}{2}\lss(P)$.
The bound is attained at $P$ for which the bound of Theorem~\ref{T:ls(P)2A} is attained and also $\ls(P)=\lss(P)$. We conclude  that such $P$ are lattice-equivalent to $\conv\{(0,0), (h,0), (0,1)\}$. 
\end{proof}

\section{Classification of minimal  lattice polygons $P$ of fixed  $\lss(P)$.}

\begin{definition} We say that that a lattice polygon $P$ with $\lss(P)=h$ is minimal if there is no lattice polygon $P'$ properly contained in $P$ such that $\lss(P')=h$.
\end{definition}

In this section we will classify all the minimal lattice polygons $P$ of fixed lattice size $\lss(P)=h$. This classification will provide an alternative proof for Corollary~\ref{C:lss(P)2A}.

\begin{proposition}\label{P:triangle} For  integers $a,b\in [1,h-1]$ define $T=\conv\{(0,0), (a,h), (h,b)\}$, as depicted in Figure~\ref{F:triangle-in-square}. Then  $\lss(T)=h$. Such $T$ is minimal if and only if $a+b\geq h$.
\end{proposition}
\begin{figure}[h]
\begin{center}
\includegraphics[scale=.65]{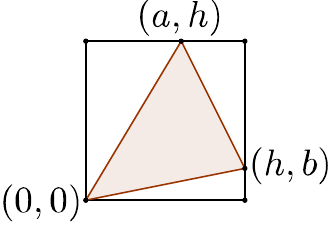} 
\caption{Proposition~\ref{P:triangle}}
\label{F:triangle-in-square}
\end{center}
\end{figure}

\begin{proof}
The first claim follows from Proposition~\ref{P:touch}.
To prove the second claim, suppose first that $a+b\geq h$. We need to show that for any lattice polygon $P$ properly contained in $T$ we have $\lss(P)<h$. 
We have 
$$\w_{(1,-1)}(T)=(h-b)-(a-h)=2h-a-b\leq h.
$$ 
Note that if we drop $(a,h)$ from $T$ to get to $P$ (see Definition~\ref{D:drop}) then $\w_{(1,-1)}(P)< h$. Also, $\w_{(0,1)}(P)< h$
and hence $\begin{bmatrix}1&-1\\0&1\end{bmatrix}$ composed with a lattice translation maps $P$ inside  $[0,h-1]^2$. The case of dropping $(h,b)$ is similar.
 If we drop $(0,0)$ then $P\subset [0,h-1]^2$. 

Next assume that $a+b\leq h-1$. If we drop $(h,b)$ to get from $T$ to $P$ then $(h-1,b)$ is in $P$ and we get $\w_{(1,0)}(P)=h-1$,  $\w_{(0,1)}(P)=h$, 
$$\w_{(1,-1)}(P)=h-1-b-(a-h)=2h-1-a-b\geq h.
$$
Also, $\w_{(1,1)}(P)\geq h+a>h$ and hence the standard basis is reduced and by Theorem~\ref{T:reduced_computes} we have  $\lss(P)=h$, so $T$ is not minimal.
\end{proof}

\begin{proposition}\label{P:quad}  For integers $a,b,c,d\in [1,h-1]$ define $$Q=\conv\{(a,0), (0,b), (h,h-c), (h-d,h)\},$$ 
as depicted in Figure~\ref{F:quad-in-square}. Then we have $\lss(Q)=h$. Such $Q$ is minimal if and only if 
$$
\min\{a,b\}+\min\{c,d\}>h\ \ {\rm or}\ \ \max\{a,c\}+\max\{b,d\}<h.
$$
Furthermore, if $Q$ satisfies one of these inequalities it cannot satisfy the other. Also, if $Q$ satisfies one of them, it is lattice-equivalent to a quadrilateral that satisfies the other.
\end{proposition}
\begin{figure}[h]
\begin{center}
\includegraphics[scale=.65]{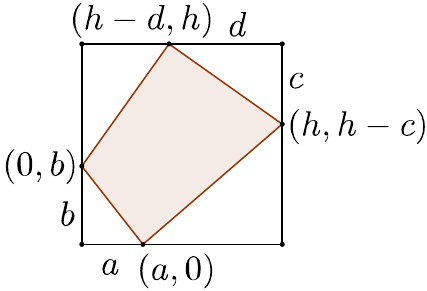} 
\caption{Proposition~\ref{P:quad}}
\label{F:quad-in-square}
\end{center}
\end{figure}
\begin{proof} Since $P$ touches all four sides of $[0,h]^2$,  by Proposition~\ref{P:touch}, we have $\lss(P)=h$.
Now assume  that one of the two inequalities is satisfied. Then
$$\w_{(1,1)}(Q)=\max\{2h-c,2h-d\}-\min\{a,b\}=2h-(\min\{a,b\}+\min\{c,d\})<h
$$
or 
$$\w_{(1,-1)}(Q)=\max\{a,c\}+\max\{b,d\}<h.
$$
If we drop one of the vertices from $Q$ to get $P$, we would also have $\w_{(1,0)}(P)<h$ or $\w_{(0,1)}(P)<h$ and hence we would be able to use one of $\begin{bmatrix}1&0\\1&\pm 1\end{bmatrix}$, $\begin{bmatrix}0&1\\1&\pm 1\end{bmatrix}$ and a lattice translation to fit $P$ into a smaller square, so $Q$ is minimal.

If neither of the two inequalities holds, we have $\w_{(1,1)}(Q)\geq h$ and $\w_{(1,-1)}(Q)\geq h$. If both of these inequalities are strict, then after we drop a vertex to get from $Q$ to $P$ we would have $\w_{(1,1)}(P)\geq h$ and 
$\w_{(1,-1)}(P)\geq h$, so the standard basis would still be reduced. We would also have either $\w_{(1,0)}(P)=h$ or $\w_{(0,1)}(P)=h$, so we can conclude that $\lss(P)=h$ and hence $Q$ is not minimal.

A similar argument works if we have $\w_{(1,1)}(Q)=h$ and $\w_{(1,-1)}(Q)\geq h$, or $\w_{(1,1)}(Q)\geq h$ and $\w_{(1,-1)}(Q)=h$. We would  need to drop a vertex that does not change the width in the direction $(1,1)$ in the first case and in the direction of $(1,-1)$ in the second.

If we have $\w_{(1,1)}(Q)=h$ and $\w_{(1,-1)}(Q)=h$ then $\max\{a,c\}+\max\{b,d\}=h$ and $\min\{a,b\}+\min\{c,d\}=h$. Since we have $\min\{a,b\}\leq \max\{b,d\}$ and $\min\{c,d\}\leq \max\{a,c\}$ we can conclude
$\min\{a,b\}= \max\{b,d\}$ and $\min\{c,d\}= \max\{a,c\}$. Hence we get $a=b=c=d=h/2$, and then  $Q$ contains a horizontal segment of lattice length $h$, so $Q$ is not minimal. 

Note that $Q$ cannot satisfy both inequalities since then we would get $$h<\min\{a,b\}+\min\{c,d\}\leq\max\{a,c\}+\max\{b,d\}<h.$$

If $Q$ satisfies the second inequality then the transformation $(x,y)\mapsto (h-x,y)$ maps $Q$ to 
$\conv\{(h-a,0), (0,h-c), (h,b), (d,h)\}\subset [0,h]^2,$
and its parameters $a'=h-a$, $b'=h-c$, $c'=h-b,d'=h-d$ satisfy
\begin{align*}\min\{a',b'\}+\min\{c',d'\}&=\min\{h-a,h-c\}+\min\{h-b,h-d\}\\
&=2h-\max\{a,c\}-\max\{b,d\}>h.
\end{align*}
\end{proof}

\begin{Theorem}\label{T:minimal} Let $P\subset\mathbb{R}^2$ be a minimal lattice polygon with $\lss(P)=h$. Then $P$ is lattice-equivalent to one of the following:
\begin{itemize}
\item[(a)] The segment $I=[(0,0),(h,0)]$;
\item[(b)]  The triangle $T=\conv((0,0), (a,h), (h,b))$, where $a$ and $b$ are integers that satisfy $a,b\in[1,h-1]$ and  $a+b\geq h$; 
\item[(c)]  The quadrilateral $Q=\conv\{(a,0), (0,b), (h,h-c), (h-d,h)\}$,  where $a, b, c,$ and $d$ are integers that satisfy $a,b,c,d\in[1,h-1]$ and $\min\{a,b\}+\min\{c,d\}> h$.
\end{itemize}
\end{Theorem}

\begin{proof} We can assume that $P\subset [0,h]^2$. Suppose first  that $P$ has lattice points on all four sides of $[0,h]^2$. One way this can happen is  when $P$ contains a segment connecting two opposite vertices of the square. 
Then, since  by Example~\ref{E:ex-I} we have $\lss(I)=h$ and $P$ is minimal, we conclude that $P$ is lattice-equivalent to $I$. Next, $P$ could be  a lattice triangle, one of whose vertices is a vertex of $[0,h]^2$ and two other vertices are on the adjacent sides of the square.
In this case, by Proposition~\ref{P:triangle}, $P$ is lattice-equivalent to $T=\conv((0,0), (a,h), (h,b))$ with $a,b\in[1,h-1]$ and  $a+b\geq h$. Finally, $P$ could be a quadrilateral  with exactly one vertex on each side of $[0,h]^2$ and, in this case, by Proposition~\ref{P:quad}, $P$ is lattice-equivalent to $$Q=\conv((a,0), (0,b), (h,h-c), (h-d,h))$$ with integers $a,b,c,d\in[1,h-1]$ that satisfy $\min\{a,b\}+\min\{c,d\}>h$.

Suppose next that the standard basis is reduced and hence $P\subset [0,h]^2$. If $P$ touches all four sides of $[0,h]^2$, we are done by the above. If $P$ touches only three sides, we can assume, switching the basis vectors,   that $\w_{(1,0)}(P)=h$, $\w_{(0,1)}(P)<h$, and also $\w_{(1,\pm 1)}(P)\geq h$. If $\w_{(1,1)}(P)=h$ or $\w_{(1,-1)}(P)=h$ we can use one of  $\begin{bmatrix}1&0\\ 1&\pm 1\end{bmatrix}$ to reduce to the case of $P$ touching all four sides of $h\square$. Hence we can assume that $\w_{(1,\pm 1)}(P)>h$. If $P$ contains the entire segment $I=[(0,0),(h,0)]$, by the minimality of $P$ and by Example~\ref{E:ex-I} we have $P=I$. Otherwise, we can assume that $(h,0)\not\in P$. Let  $(a,0)$ with $a<h$ be the rightmost point of $P$ in  $[(0,0),(h,0)]$ and assume that $a>0$. We drop $(a,0)$ to get from $P$ to $P'$, see the first diagram in Figure~\ref{F:three-sides}. Since $0<a<h$ we have $(a,1)\in P$ and hence the width in the directions $(1,\pm 1)$ could drop by at most 1. Hence we have  $\w_{(1,0)}(P')=h$,
$\w_{(1,\pm 1)}(P')\geq h$ and we can conclude that $\lss(P')=h$, so $P$ is not minimal.

\begin{figure}[h]
\begin{center}
\includegraphics[scale=.65]{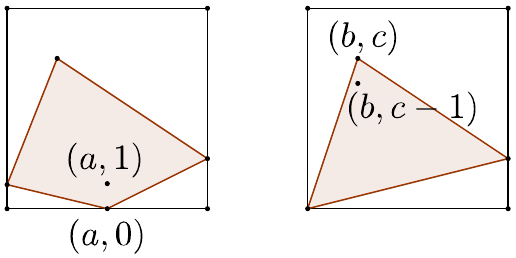} 
\caption{$P$ touches three sides of $[0,h]^2$}
\label{F:three-sides}
\end{center}
\end{figure}

Finally, let $a=0$. Suppose that the highest point of $P$ is on the line $y=c$ and let $(b,c)$ be the leftmost point of $P$ on this line. We then drop $(b,c)$ to get from $P$ to $P'$. If $c>1$ and $b< h$ then $(b,c-1)\in P$ and, as above, we conclude that $P$ is not minimal. If $c=1$ or $b=h$, with the exception of the case $(b,c)=(0,1)$, we have $P\subset\conv\{(0,0), (h,0), (h,h)\}$, but then $$\w_{(1,-1)}(P)\leq \w_{(1,-1)}\conv\{(0,0), (h,0), (h,h)\}=h.$$
If $(b,c)=(0,1)$ then $(h,1)\in P$ and $P$ is not minimal.
\end{proof}

This classification leads to an alternative argument for Corollary~\ref{C:lss(P)2A}.  If $P$ is of nonzero area and contains a lattice segment of lattice length $h$,  then $A(P)\geq h/2$ and in this case this inequality turns into equality if and only if $P$ is lattice-equivalent to $\conv\{(0,0), (h,0), (0,1)\}$.
It remains to show that the strict form of this inequality  holds  for triangles $T$ and quadrilaterals $Q$ from Theorem~\ref{T:minimal}.
For triangle $T$, we get $A(T)=(h^2-ab)/2> h/2$ since $ab< h(h-1)$ as $a,b\leq h-1$.

For quadrilateral $Q$, the inequality $2A(Q)> h$ rewrites as
$$
2h^2-ab-cd-(h-a)(h-c)-(h-b)(h-d)=(h-a)(b+c)+h(a+d)-d(b+c)>h\\
$$

Reflecting in the line $y=x$, if necessary, we can assume that $a+d\geq b+c$, so 
$$h(a+d)-d(b+c)\geq (h-d)(b+c)\geq (b+c).$$ 
Adding this up with  $(h-a)(b+c)\geq b+c$ we get
\begin{align*}
h(a+d)-d(b+c)+(h-a)(b+c)\geq 2(b+c)\geq 2(\min\{a,b\}+\min\{c,d\})>2h, 
\end{align*}
and this completes the argument.

\subsection*{Acknowledgments} 
We are grateful to Gennadiy Averkov for pointing us to~\cite{TothMakai} and for explaining that Theorem~\ref{T:whbound} is a corollary of a bound by Fejes-T\'oth and Makai. 
We also would like no thank Ivan Soprunov for reading an early version of this manuscript and providing many helpful comments.

\end{document}